\documentclass{amsart}
\usepackage[utf8]{inputenc}
\usepackage{amssymb,amsmath,color,float}

\usepackage{placeins}

\newcommand{\End}{\textup{End}}

\newcommand{\A} {\mathbb{A}}

\newcommand{\R}{\mathbb{R}}
\newcommand{\Z}{\mathbb{Z}}
\renewcommand{\P}{\mathbb{P}}

\newcommand{\dist}{\mbox{dist}}

\renewcommand{\L}{\mathcal{L}}

\newtheorem{theorem}{Theorem}[section]

\newtheorem{proposition}[theorem]{Proposition}
\newtheorem{conjecture}[theorem]{Conjecture}
\newtheorem{definition}[theorem]{Definition}
\newtheorem{remark}[theorem]{Remark}

\newif\ifhascomments \hascommentstrue
\ifhascomments
  \newcommand{\matt}[1]{{\color{red}[[\ensuremath{\spadesuit\spadesuit\spadesuit} #1]]}}
  \newcommand{\david}[1]{{\color{red}[[\ensuremath{\star\star\star} #1]]}}
  \newcommand{\chris}[1]{{\color{red}[[\ensuremath{\clubsuit\clubsuit\clubsuit} #1]]}}
\else
  \newcommand{\matt}[1]{}
  \newcommand{\david}[1]{}
  \newcommand{\chris}[1]{}
\fi

\title{Curves of best approximation on wonderful varieties}
\author{Christopher Manon, David McKinnon, and Matthew Satriano}
\date{August 2022}

\begin{document}

\begin{abstract}
    We give an unconditional proof of the Coba conjecture for wonderful compactifications of adjoint type for semisimple Lie groups of type $A_n$.  We also give a proof of a slightly weaker conjecture for wonderful compactifications of adjoint type for arbitrary Lie groups.
\end{abstract}
    
	\maketitle
	
	\section{Introduction}
	
	The distribution of rational points on algebraic varieties has intensely interested the mathematical community for millennia, though it has not always been phrased in such modern terms.  In this paper, we address the problem of how well a rational point on an algebraic variety can be approximated by nearby rational points.

    The driving force behind this study is a conjecture of the second author from 2007 (\cite{M}, Conjecture 2.7, reproduced here as Conjecture~\ref{conj:ratcurve}).  It says, roughly speaking, that as long as there are enough rational points near a rational point $P$, then there is a curve that contains the best rational approximations to $P$.  This curve is called a curve of best approximation to $P$.

    This conjecture has been verified in a large number of cases, and is known to follow from Vojta's Conjecture in a much larger number of cases.  In \cite{LMS}, the authors describe a framework for deducing from Vojta's Conjecture a slightly weaker conjecture (Conjecture~\ref{conj:nearratcurve}) which states there is a rational curve $C$ through $P$ such that rational approximations to $P$ along $C$ are better than any Zariski dense sequence of rational points approximating $P$.
    %and actually prove that implication in dimension two.  
    A Noetherian induction argument can deduce the stronger Conjecture~\ref{conj:ratcurve} from the weaker Conjecture~\ref{conj:nearratcurve} for general varieties (not necessarily smooth), and in dimension two the conjectures are equivalent.  Results of Monahan and the third author (\cite{MoSat}) as well as results of the second and third authors (\cite{MS}) deduce the weaker conjecture from Vojta for horospherical and toric varieties, respectively. Casta\~neda (\cite{Ca}), McKinnon (\cite{M} and \cite{MC}), and McKinnon-Roth (\cite{MR}) prove the conjecture unconditionally for various rational projective surfaces and all cubic hypersurfaces of dimension at least two.   

    In this paper, we prove Conjecture~\ref{conj:nearratcurve} {\em unconditionally} for wonderful compactifications $X$ of adjoint type of semisimple Lie groups, for an arbitrary nef divisor class.  This represents the first unconditional proof of cases of the conjecture in arbitrarily high dimension and degree.  The proof proceeds by finding a curve $C$ of small anticanonical degree through $P$.  We compute the dimension of the space of global sections of a generator of the nef cone using the representation theory of the group $G$.  The Liouville type theorem (Theorem~3.3) from \cite{MR} then gives a lower bound for $\alpha$ for Zariski dense sequences in terms of this dimension.  The proof concludes by showing the degree of $C$ is low enough to contain a sequence whose approximation constant is lower than the lower bound obtained from the Liouville-type theorem for dense sequences.

    In Section \ref{prelim}, we describe some preliminaries necessary for the proof, including a brief discussion of the geometry of wonderful compactifications and definitions of the basic notions of Diophantine approximation.  In Section \ref{mainthm} we prove the main theorem.  Section \ref{tables} contains tables describing the dimensions of the spaces of global sections for generators of the nef cone and the intersection numbers of the longest root curve with those same generators.

    \section{Preliminaries}\label{prelim}

    \subsection{The Coba Conjecture}\label{sec-coba}

For the rest of the paper we fix a number field $k$ and a place $v$ of $k$.  The place $v$ may be archimedean or not.

The first step in stating the Coba conjecture is to define the approximation constant.  We first define the approximation constant for a sequence $\{x_i\}$ of $k$-rational points converging ($v$-adically) to $P$.  The tension in the definition is between the height of the approximations $x_i$ and the distance of the $x_i$ to $P$.  This tension is measured by the product $H(x_i)\dist(x_i,P)^\gamma$.  Setting $\gamma=0$ will make this product approach infinity, while increasing $\gamma$ will make that product smaller for large enough $i$.  The smallest $\gamma$ that makes the product converge to zero is the approximation constant.  

More precisely, we have the following definition.

\begin{definition}\label{dfn:seqappconst}
	Let $X$ be a projective variety, $P\in X(\overline{k})$, $L$ a $\mathbb{Q}$-Cartier divisor on $X$.  For any sequence $\{x_i\}\subseteq X(k)$ of distinct
	points with $\dist(P,x_i)\rightarrow 0$, we set
	$$A(\{x_i\}, L) = \left\{{
		\gamma\in\R \mid
		\dist(P,x_i)^{\gamma} H_{L}(x_i)\,\,\mbox{is bounded from above}
	}\right\}.
	$$
	If $\{x_i\}$ does not converge to $P$ then we set $A(\{x_i\},L)=\emptyset$.
\end{definition}

It follows easily from the definition that if $A(\{x_i\}, L)$ is nonempty then it is an
interval unbounded to the right.  That is, if $\gamma\in A(\{x_i\},L)$ then $\gamma+\delta\in A(\{x_i\},L)$ for any $\delta>0$ -- if a choice of $\gamma$ makes the product converge to zero, then making $\gamma$ bigger won't change that.

The approximation constant is the left endpoint of the interval $A(\{x_i\},L)$.

\begin{definition}
	For any sequence $\{x_i\}$ we set $\alpha(\{x_i\},P,L)$ to be the infimum of $A(\{x_i\},L)$.
	In particular, if $A(\{x_i\},L)=\emptyset$ then $\alpha(\{x_i\},P,L)=\infty$.
	We call $\alpha(\{x_i\},P,L)$ the approximation constant of $\{x_i\}$ with respect to $L$.
\end{definition}

This allows us to define the approximation constant of $P$.  

\begin{definition}\label{dfn:approx}
	The approximation constant $\alpha(P,L)$ of $P$ with respect to
	$L$ is defined to be the infimum of all approximation constants of
	sequences of points in $X(k)$ converging $v$-adically to $P$.
	If no such sequence exists, we set $\alpha(P,L)=\infty$.
\end{definition}

With the definition of the approximation constant in hand -- and the concurrent quantitative knowledge of how good an approximation is -- we can state the conjectures.
    
       \begin{conjecture}[Coba conjecture]\label{conj:ratcurve}
           Let $X$ be a smooth variety defined over a number field $k$, and $P\in X$ a $k$-rational point.  Let $A$ be an ample line bundle on $X$.  If $\alpha(P,A)<\infty$, then there is a curve of best $A$-approximation to $P$.  In other words, there is a curve $C$ such that
            \[\alpha(P,A|_C)=\alpha(P,A)\]
            where the first $\alpha$ is computed on $C$, and the second is computed on $X$.
        \end{conjecture}

        \begin{conjecture}[near-Coba conjecture]\label{conj:nearratcurve}
           Let $X$ be a smooth variety defined over a number field $k$, and $P\in X$ a $k$-rational point.  Let $A$ be an ample line bundle on $X$.  If $\alpha(P,A)<\infty$, then there is a curve $C$ such that for any Zariski dense sequence of $k$-rational points $\{x_i\}$, we have
            \[\alpha(P,A|_C)\leq\alpha(P,\{x_i\})\]
            In other words, $C$ approximates $P$ at least as well as any Zariski dense sequence of $k$-rational points.
        \end{conjecture}

In both conjectures, the hypothesis $\alpha(P,A)<\infty$ is equivalent to $\alpha(P,L)<\infty$ for {\em any} ample divisor $L$, and to $\alpha(P,L)<\infty$ for {\em every} ample divisor $L$.  This is a consequence of the existence of two positive constants $C_1$ and $C_2$ such that $C_1L-A$ and $C_2A-L$ are effective.

We end this subsection with a result that gives lower bounds on the approximation constant of a Zariski dense sequence.

\begin{proposition}\label{prop:lower-bound-alpha-Zar-dense}
    
    Let $D$ be a divisor on an $n$-dimensional variety $X$ over a number field $K$.  If 
    \[h^0(D)>\left(\begin{array}{c} n+d-1 \\ n \end{array}\right),\]
    then for every smooth point $P\in X(K)$ and every Zariski dense sequence $\{x_i\}$ of rational points, we have
    \[
    \alpha(\{x_i\},P,D)\geq d.
    \]
\end{proposition}
\begin{proof}
    A global section of $D$ vanishes to order $d$ at $P$ if and only if, in some (every) system of local coordinates, all the terms of degree at most $d-1$ are zero.  The space of monomials of degree at most $d-1$ in $n$ variables has dimension ${n+d-1\choose n}$.  The lower bound for $h^0(D)$ ensures that there will be a nonzero section of $D$ that vanishes at $P$ to order at least $d$.  The proposition then follows from the Liouville-type Theorem~3.3 of \cite{MR}.
\end{proof}

\begin{remark}
In fact, the Liouville-type theorem from \cite{MR} allows us to say somewhat more than this.  The lower bound for $\alpha(\{x_i\},P,D)$ holds for any sequence $\{x_i\}$ that does not eventually lie inside the (push-forward of) the stable base locus of the divisor $\pi^*D-dE$, where $E$ is the exceptional divisor of the blowup $\pi\colon\tilde{X}\to X$.  However, this stable base locus may be difficult to compute, so in this paper we will only use Proposition~\ref{prop:lower-bound-alpha-Zar-dense} as stated.
\end{remark}

    \subsection{Geometry of the wonderful compactification}\label{sec-geometry}

    Let $\tilde G$ be a simply connected simple group of rank $r$. We fix a maximal torus $T \subset \tilde G$ and a set of positive simple roots $\alpha_1 \ldots \alpha_r$.  This information determines a weight lattice $\Lambda = \textup{Hom}(T, \mathbb{G}_m)$ and root lattice $\mathcal{R} = \Z\{\alpha_1, \ldots, \alpha_r\} \subset \Lambda$, along with the monoid of dominant weights $\Lambda_+ \subset \Lambda$. We let $\omega_1, \ldots, \omega_r$ denote the fundamental dominant weights. Recall that weights are ordered by the dominant weight ordering: we say $\eta \prec \lambda$ if $\lambda - \eta$ is a positive sum of simple positive roots. The dominant weights index the irreducible representations of $\tilde G$, and we let $V_\lambda$ denote the irreducible corresponding to $\lambda \in \Lambda_+$. Taking dual representations gives rise to an involution $\lambda \to \lambda^*$, where $V_{\lambda^*} = V_\lambda^*$.

    The center $Z(\tilde G)$ can be identified with the quotient group $\Lambda/\mathcal{R}$. The quotient $\tilde G/Z(\tilde G)$ is the \emph{adjoint form} of $\tilde G$. The adjoint form is the focus of our main result, so we denote it by $G$.

    The \emph{wonderful compactification} $G \subset X$ is a smooth, projective variety with simple normal crossings boundary $\partial X = X \setminus G$. The action of $\tilde G\times \tilde G$ on $G$ extends to $X$, and the irreducible components $\partial X = W_1 \cup \cdots \cup W_r$ are the closures of the orbits of $\tilde G\times \tilde G$. Let $W_S = \bigcap_{i \in S} W_i$ for $S \subset [r]$.  Each $W_S$ is a $\tilde G\times \tilde G$ space with a dense, open orbit $W_S^\circ \subset W_S$.  Following \cite[2.2]{DeConciniProcesi}, the map from the maximal torus $T \to G$ extends to an embedding of $\A^r$ into $X$.  The intersection $\A^r \cap W_S^\circ$ is the open $T$ orbit in the intersection of the coordinate hyperplanes corresponding to $i \in S$.  In particular, any point in $W_S^\circ$ is in the $\tilde G\times \tilde G$ orbit of the distinguished point in $Id_S \in \A^r \cap W_S^\circ$ in this orbit.  

    The Picard group $\textup{Pic}(X)$ is isomorphic to the weight lattice $\Lambda$ (see \cite[Example 2.2.4]{Brion}). The effective cone $\textup{Eff}(X) \subset \textup{Pic}(X)\otimes \mathbb{Q}$ has extremal rays generated by the simple positive roots $\alpha_1, \ldots, \alpha_r$, while the nef cone $\textup{Nef}(X) \subset \textup{Pic}(X)\otimes \mathbb{Q}$ has extremal rays generated by the fundamental dominant weights $\omega_1, \ldots, \omega_r$. The monoid of nef line bundles is identified with $\Lambda_+$. Accordingly we let $\L_\lambda$ denote the bundle associated to the dominant weight $\lambda \in \Lambda_+$ (\cite[Example 2.3.5]{Brion}).
    
    The space of sections of $\L_\lambda$ has the following description as a $\tilde G\times \tilde G$ representation: \[H^0(X, \L_\lambda) = \bigoplus_{\eta \prec \lambda} \textup{End}(V_\eta).\] In particular, $H^0(X, \L_\lambda)$ has a distinguished summand $\textup{End}(V_\lambda)$ which defines a morphism $\phi_\lambda: X \to \P\textup{End}(V_{\lambda^*}).$ This map can be computed directly by taking the closure of the $\tilde{G} \times \tilde{G}$ orbit of the identity element of $\textup{End}(V_{\lambda^*})$. The geometry of these morphisms is explored in \cite{DeConciniProcesi}.

	\section{Main Theorem for wonderful compactifications of adjoint type}\label{mainthm}
	
	\begin{theorem}\label{adjoint}
		Let $X$ be the wonderful compactification of adjoint type of a semi-simple Lie group $G$, defined over a number field $k$ and let $D$ be a nef divisor on $X$.  Let $P$ be any $k$-rational point on $X$.  Then any sequence of best $D$-approximation to $P$ is not Zariski dense.

        If $G$ is of type $A_n$ or $C_n$ and $P$ is in the open orbit then the \ref{conj:ratcurve} is true for $X$.
	\end{theorem}

	\medskip

	\noindent
	{\it Proof:} \/  
    We begin by considering a simple group $G$.  We consider points in the dense, open orbit $G \subset X$.  Without loss of generality, we may assume that $P \in G$ is the identity element $Id \in G$. Let $\theta: \A^1 \to G$ be the one-parameter subgroup corresponding to the longest root, and let $C_\theta$ denote the closure of the action of $\theta$ on $Id$.  We show that we can find a sequence of best approximation on $C_\theta$ that approximates $Id$ at least as well as any Zariski dense sequence. We will use the generalization of Liouville's Theorem from \cite{MR} to give a lower bound for the approximation constant of any Zariski dense sequence, and then directly compute the approximation constant for a best sequence on $C_\theta$, finding that the latter is no greater than the former.

\vspace{1em}
    
	\noindent\emph{Step 1:~$D$ is a colour, $P\in G$, and $G$ is simple.} Assume first that $D$ is a colour.  Theorem~2.8 of \cite{MR} shows the approximation constant of a sequence of best approximation on $C_\theta$ is simply $D.C_\theta$ because $C_\theta$ is smooth at $Id$ and birational to $\P^1$ over $k$.  That is, we have:
    \[\alpha(P,D|_{C_\theta})=D.C_\theta\]
    Since $D=D_\omega$ corresponds to a fundamental dominant weight $\omega$, we may compute $D_\omega.C_\theta$ as $\langle \omega,\theta^\vee\rangle$. The second column of Table~\ref{rootcurves} lists the values of $D_\omega.C_\theta$; the multiple entries per row correspond to the different choices of fundamental dominant weight, ordered in the standard manner.
    
	What remains is to show that if $\{x_i\}$ is any Zariski dense sequence, then: 
    \[\alpha(\{x_i\},P,D)\geq \alpha(P,D|_{C_\theta})=D_\omega.C_\theta\]  
	By Proposition \ref{prop:lower-bound-alpha-Zar-dense}, we need only verify that
	\[\binom{\dim(X) - 1 + D_\omega \cdot C_\theta}{D_\omega \cdot C_\theta - 1} \leq h^0(D_\omega).\]
    In fact, since $h^0(D_\omega)\geq\dim\End(V_\omega)$, we verify the stronger statement that
    \[\binom{\dim(X) - 1 + D_\omega \cdot C_\theta}{D_\omega \cdot C_\theta - 1} \leq \dim\End(V_\omega).\]
    Table~\ref{rootcurves} lists these values for groups of simple type.  A comparison with the values of $\dim\End(V_\omega)$ in Table~\ref{dimensioncounts} shows that in every case, $\dim\End(V_\omega)$ is large enough to force the necessary order of vanishing, and often larger than necessary.  To pick an extreme example, notice that for the third fundamental weight of $E_8$, the dimension of $\dim\End(V_\omega)$ needs to be greater than 2573000 to conclude the result, but the actual dimension of $\End(V_\omega)$ is
	\[47596949737616581696\]
	{\em Thirteen} orders of magnitude separate those two numbers.  To put it another way, the approximation constant of a Zariski dense sequence is at least 11, while a sequence of best approximation along the root curve is a mere 4.  

    \vspace{1em}
    
	\noindent\emph{Step 2:~$D$ is nef, $P\in G$, and $G$ is simple.} 
        We have so far shown that if $D$ is a colour, then the root curve $C_\theta$ contains a sequence of smaller $D$-approximation to $P$ than that of any Zariski dense sequence.  Since the colours generate the nef cone, we can write an arbitrary nef divisor class $D$ as a non-negative linear combination of colours.  By Corollary~3.2 of \cite{M}, since the same curve $C_\theta$ approximates $P$ better than any dense sequence with respect to all the colours, it does the same with respect to $D$.  

    \vspace{1em}
    
	\noindent\emph{Step 3:~$D$ is nef, $P$ is arbitrary, and $G$ is simple.} 
    Now assume that $P$ lies on an open component of the boundary $W_S^\circ$ (see Section \ref{sec-geometry}).  Using the $G\times G$ action, without loss of generality we may assume that $P = Id_S$ for some $S \subset [r]$.  We choose a one-parameter subgroup $\eta: \mathbb{G}_m \to T$ such that the limit of the action on $Id$ is the point $Id_S$.  We consider the map $\tilde{\eta}: \A^1\times \P^1 \to X$ defined by $\{\eta(t)C_\theta\}_{t\in\A^1}$. This is a one-parameter family of rational curves parameterized by $\A^1$.  When $t=0$, by definition of $\eta$ the resulting curve $C_{S,\theta}$ is the closure of the action of the longest root group for $\theta$ through $Id_S$. Since all the curves in the family have the same intersection properties, we may again deduce from Tables \ref{rootcurves} and \ref{dimensioncounts} that the $1$-cycle $C_{S,\theta}$ has small enough degree that the component through $Id_S$ has a smaller $L$-approximation constant than any Zariski dense sequence, for any nef $L$.  This concludes the argument in the case that $G$ is simple.

    \vspace{1em}
    
	\noindent\emph{Step 4:~full conjecture for types $A_n$ and $C_n$.} If $G$ has type $A_n$ or $C_n$, then the root curve $C_\theta$ has intersection number 1 with every colour.  By Proposition 2.14 (b) of \cite{MR}, this means that $C_\theta$ is a curve of best approximation for every colour, except possibly for sequences contracted by that colour.  Since the colours generate the nef cone, we deduce that $C_\theta$ is a curve of best approximation for an arbitrary nef divisor $D$, except possibly for sequences contracted by some colour.  Since the colours are basepoint free and their associated morphisms are isomorphisms on the open orbit, we conclude that Conjecture \ref{conj:ratcurve} is true for $P$ with respect to an arbitrary nef divisor $D$.

    \vspace{1em}
    
	\noindent\emph{Step 5:~$D$ is nef, $P$ is arbitrary, and $G$ is semisimple.} Finally, if $G$ is semisimple of adjoint type, we can write $G=G_1\times\ldots\times G_r$ for simple groups $G_i$ of adjoint type.  The wonderful compactification of $G$ is the product of the wonderful compactifications of the $G_i$.  We already know that the point $\mbox{Id}_1$ on $G_1$ is approximated by a sequence $\{x_i\}$ on the curve $C_{\theta_1}$ better than any Zariski dense sequence.  It is therefore trivially true that the sequence $\{(x_i,P_2,\ldots,P_r)\}$ approximates the point $P=(P_1,\ldots,P_r)$ better than any Zariski dense sequence on $X$.  \qed
    
	\section{Tables}\label{tables}

	\FloatBarrier
	
	\begin{center}
		\begin{table}[H]\Large
			\begin{tabular}{ |c|c| c | c|} \hline
				$G$ & $D_\omega \cdot C_\theta = \langle \omega, \theta^\vee \rangle$ & $\binom{\dim(X) - 1 + D_\omega \cdot C_\theta}{\dim(X)}$ \\ 
				\hline
				$A_n$ & 1, \ldots, 1 & 1,\ldots,1\\ 
				%$\SL_{n+1}$ 
				\hline
				
				$B_n$ & 1, 2,\ldots, 2, 1 & $1, n(2n+1),\ldots,n(2n+1),1$ \\ 
				%$\SO_{2n+1}$ 
				\hline
				
				$C_n$ &  1, \ldots, 1 & 1,\ldots,1 \\ 
				%$\Sp_{2n}$
				\hline
				
				$D_n$ &  1, 2, \ldots, 2, 1, 1 & $1, n(2n-1),\ldots,n(2n-1), 1,1$\\ 
				%$\SO_{2n}$
				\hline
				
				$E_6$ &  1, 2, 3, 2, 1, 2 & 1, 78, 3081, 78, 1, 78 \\ 
				
				\hline
				
				$E_7$ &  1, 2, 3, 4, 3, 2, 2 & 1, 133, 8911, 400995, 8911, 133, 133\\ 
				
				\hline
				
				$E_8$ &  2,3,4,5,6,4,2,3 & 248, 30876, 2573000, 161455750, \\
                & & 8137369800, 2573000, 248, 30876 \\ 
				
				\hline
				
				$F_4$ &  2,3,2,1 & 52, 1378, 52, 1\\ 
				
				\hline
				
				$G_2$ &  1,2 & 1, 14 \\ 
				
				\hline
				
			\end{tabular}\vspace*{.1in}
			\caption{Intersection numbers of the root curve}\label{rootcurves}
		\end{table}
	\end{center}
    
	\begin{center}
		\begin{table}[H]\Large
			\begin{tabular}{ |c|c| c| } \hline
				$G$ & $\dim(X)$ & $\dim\End(V_\omega) \leq h^0(X, D_\omega)$ \\ 
				\hline
				$A_n$ & $n(n + 2)$ & $(n+1)^2, \ldots, \binom{n+1}{k}^2, \ldots, (n+1)^2$ \\ 
				%$\SL_{n+1}$ 
				\hline
				
				$B_n$ & $n(2n+1)$ & $(2n+1)^2, \ldots, \binom{2n+1}{k}^2, \ldots, \binom{2n+1}{n}^2$ \\ 
				%$\SO_{2n+1}$ 
				\hline
				
				$C_n$ & $n(2n+1)$ & $(2n)^2, \ldots, (\binom{2n}{k} - \binom{2n}{k-2})^2, \ldots, (\binom{2n}{n} - \binom{2n}{n-2})^2$ \\ 
				%$\Sp_{2n}$
				\hline
				
				$D_n$ & $n(2n-1)$ & $(2n)^2, \ldots, \binom{2n}{k}^2, \ldots, \binom{2n}{n-1}^2, (\frac{1}{2}\binom{2n}{n})^2$ \\ 
				%$\SO_{2n}$
				\hline
				
				$E_6$ & $78$ & $27^2, 351^2, 2925^2, 352^2, 27^2, 78^2$ \\ 
				
				\hline
				
				$E_7$ & $133$ & $133^2, 8645^2, 365750^2, 27664^2, 1539^2, 56^2, 912^2$ \\ 
				
				\hline
				
				$E_8$ & $248$ & $3875^2, 6696000^2, 6899054264^2, 146325270^2,$ \\
                & & $2450240^2, 30380^2, 248^2, 147250^2$ \\ 
				
				\hline
				
				$F_4$ & $52$ & $26^2, 52^2, 273^2, 1274^2$ \\ 
				
				\hline
				
				$G_2$ & $14$ & $7^2, 14^2$ \\ 
				
				\hline
				
			\end{tabular}\vspace*{.1in}
			\caption{Dimension counts}\label{dimensioncounts}
		\end{table}
	\end{center}
	
	\normalsize

	\bibliographystyle{alpha}
	\bibliography{wonderful-07-04-25}

\end{document}